\documentclass[11pt,oneside]{amsart}

      \usepackage{amssymb}

      \theoremstyle{plain}
      \newtheorem{theorem}{Theorem}[section]
      \newtheorem{lemma}[theorem]{Lemma}
      \newtheorem{corollary}[theorem]{Corollary}
      \newtheorem{proposition}[theorem]{Proposition}

      \theoremstyle{definition}
      \newtheorem{definition}[theorem]{Definition}
      
      \theoremstyle{remark}
      \newtheorem{remark}[theorem]{Remark}

      \makeatletter
      \def\@setcopyright{}
      \def\serieslogo@{}
      \makeatother
   
\begin{document}

%



   \author{Yijing Wu}
   \address{Department of Mathematics, the University of Texas at Austin}
   \email{yijingwu@math.utexas.edu}





   

   \title[]{fractional analogue of k-Hessian operators}


   \begin{abstract}
     Applying ideas of fractional analogue of Monge-Amp\'{e}re operator in \cite{C1} by L. Caffarelli and F. Charro, we consider an analogue of fractional k-Hessian operators expressed as concave envelopes of fractional linear operators, and reproduce the same regularity results when $k=2$. 

Under the set up of global solutions prescribing data at infinity and global barriers, the key estimate is to prove that fractional 2-Hessian operator is strictly elliptic. Then we can apply nonlocal Evans-Krylov theorem  \cite{EC1}\cite{EC2}  to prove such solutions are classical.
        \end{abstract}





   \date{\today}


   \maketitle



   \section{Introduction}
 
Monge-Amp\'{e}re operator is a special case of $k$-Hessian operators, which are defined by
\[
f_k(D^2u)(x)=(\sum_{1\leq i_1<i_2<...<i_k \leq n}{\lambda_{i_1}\lambda_{i_2}...\lambda_{i_k}})^{1/k},
\]
for $k$ integer and $1\leq k \leq n$. Here $\lambda_1, \lambda_2, ...,\lambda_n$ are eigenvalues of the matrix $D^2u(x)$, and $f_k$ is concave and elliptic \cite{KH1} \cite{KH2} of $\lambda$ when 
\[
\lambda=(\lambda_1, \lambda_2, ...,\lambda_n)\in \overline{ \Gamma_k}.
\]
$\Gamma_k$ is an open symmetric convex cone defined by 
\[
\Gamma_k=\{\lambda \in \mathbb{R}^n, \sigma_l(\lambda)>0,l=1,2,...,k\}.
\]
Here 
\[
\sigma_k(\lambda)=\sum_{1\leq i_1<i_2<...<i_k \leq n}{\lambda_{i_1}\lambda_{i_2}...\lambda_{i_k}}
\]
is the k-th elementary symmetric polynomial. And when $k=n$, $\Gamma_n$ is the positive cone
\[
\Gamma_n=\{\lambda \in \mathbb{R}^n, \lambda_i>0, i=1,2,...,n\}.
\]
One main ingredient of the paper \cite{C1} is the following:\\
The Monge-Amp\'{e}re equation is a concave fully nonlinear equation. If $u$ is a convex solution solving
\[
(\det{D^2u})^{1/n}(x)=g(x),
\]
then the equation is equivalent to
\[
\inf_{M\in \mathcal{M}}{L_Mu}(x)=g(x),
\]
where $L_M$ is a linear operator  defined by
\[
L_Mu(x)=trace(MD^2u(x))=\Delta(u\circ \sqrt{M})(x),
\]
and the set $\mathcal{M}$ consists of all positive symmetric matrices with determinant $n^{-n}$, independent of $x$. Moreover, the infimum is realized when $M$ is a constant multiple of the matrix of cofactor of $D^2u(x)$. \\
Then we define the fractional analogue of Monge-Amp\'{e}re equation as
\[
F_s[u](x)=\inf_{M\in \mathcal{M}}\{-C_{n,s}^{-1}(-\Delta)^s(u\circ \sqrt{M})(x)\}.
\]
Under this setting, regularity results for fractional Monge-Amp\'{e}re equation are discussed in \cite{C1}.

Therefore, it is natural to consider k-Hessian operators as concave envelopes of linear operators. We give the following definition:
\begin{definition}\label{def_fk}
As an analogue of definition of the Monge-Amp\'{e}re operator, we define
\begin{equation*}
\begin{split}
f_k(D^2u(x))&=\inf_{M\in \mathcal{M}_k}\{trace(MD^2u(x))\}\\
&=\inf_{M\in \mathcal{M}_k}\{\Delta(u\circ \sqrt{M})(x))\}\\
&=\inf_{M\in \mathcal{M}_k}\{\Delta(u(\sqrt{M}x))\}.
\end{split}
\end{equation*}
\end{definition}   

Details and explanations of the set $\mathcal{M}_k$ will be further discussed in section 2. Then we are able to give a similar definition for fractional analogues of k-Hessian operators:
\begin{definition}\label{def1}
Define fractional k-Hessian operators as
\begin{equation*}
\begin{split}
F_{k,s}[u](x)&=\inf_{M \in \mathcal{M}_k}{\{-C_{n,s}^{-1}(-\Delta)^{s}(u\circ \sqrt{M})(x)\}}\\
&=\inf_{M \in \mathcal{M}_k}{\{P.V.\int_{\mathbb{R}^n}{\frac{u(\sqrt{M}x+y)-u(\sqrt{M}x))}{|\sqrt{M}^{-1}y|^{n+2s}}\det{\sqrt{M}^{-1}}dy}\}}\\
&=\inf_{M \in \mathcal{M}_k}{\{\frac{1}{2}\int_{\mathbb{R}^n}{\frac{\delta(u,\sqrt{M}x,y)}{|\sqrt{M}^{-1}y|^{n+2s}}\det{\sqrt{M}^{-1}}dy}\}} ,
\end{split}
\end{equation*}
where 
\[
\delta(u,x,y)=u(x+y)-2u(x)+u(x-y).
\]
\end{definition}

The main idea of this article is to reproduce the regularity results of fractional Monge-Amp\'{e}re equation in \cite{C1} to fractional k-Hessian equations. 

In this article, our main purpose is to follow the ideas and set up of the paper \cite{C1}, and to prove:

(a) On each $n-1$ dimensional space, the fractional Laplacian is bounded from above and strictly positive. (Proposition \ref{prop1})

(b) When $k=2$, the operators that are close to the infimum remain strictly elliptic. (Theorem \ref{main theorem})

Here we define the strictly elliptic operator:
\begin{definition}\label{def2}
For $\epsilon_0>0$, we define a non-degenerate and strictly elliptic operator
\begin{equation*}
\begin{aligned}
F_{k,s}^{\epsilon_0}[u](x)&=\inf_{M \in \mathcal{M}_k}{\{P.V.\int_{\mathbb{R}^n}{\frac{u(\sqrt{M}x+y)-u(\sqrt{M}x))}{|\sqrt{M}^{-1}y|^{n+2s}}\det{\sqrt{M}^{-1}}dy}, \lambda_{min}(M)\geq \epsilon_0 \}}\\
&=\inf_{M \in \mathcal{M}_k}{\{\frac{1}{2}\int_{\mathbb{R}^n}{\frac{\delta(u,\sqrt{M}x,y)}{|\sqrt{M}^{-1}y|^{n+2s}}\det{\sqrt{M}^{-1}}dy},\lambda_{min}(M)\geq \epsilon_0 \}}\label{def2}.
\end{aligned}
\end{equation*}
\end{definition}

The main theorem of this article is:
\begin{theorem}\label{main theorem}
Consider $1/2<s<1$, and assume $u$ is Lipschitz continuous and semiconcave with constants $L$ and $SC$ respectively. And
\begin{equation}\label{eq1}
(1-s)F_{2,s}[u](x)\geq \eta_0
\end{equation}
for any $x \in \Omega$, in the viscosity sense for some constant $\eta_0 >0$.
Then
\begin{equation}\label{eq2}
F_{2,s}[u](x)=F_{2,s}^{\epsilon_0}[u](x)
\end{equation}
for any $x \in \Omega$ in the classical sense, with 
\[
\epsilon_0=\epsilon_0(\eta_0,n,s,L,SC)>0
\]
given by \eqref{theta}.
\end{theorem}

\begin{remark}
For simplicity, we shall assume that $0\in \Omega$ and then prove \eqref{eq2} for $x=0$. Note for the sequel that since u is semiconcave, Lemma 2.2 in paper \cite{C1} implies that $F_{2,s}(x)$ is defined in the classical sense for all $x\in \Omega$ and \eqref{eq1} holds pointwise. And this theorem states that the infimum in the definition of $F_{2,s}[u]$ cannot be realized by matrices that are too degenerate, which proves that the fractional analogue of 2-Hessian operators are locally uniformly elliptic.

\end{remark}
\begin{remark}
We can check in the proofs that $\epsilon_0$ is given by \eqref{theta}, that
\[
\epsilon_0=\sqrt{\frac{n}{n-1}}C_4^{1/s}(\frac{\mu_0}{\mu_1})^{\frac{1}{s}},
\]
with $C_4=C_4(n,s,L,SC,\eta_0)$ given by \eqref{C4}, $\mu_0$ given by \eqref{mu0} and $\mu_1$ given by \eqref{mu1}. And this shows that Theorem \ref{main theorem} is stable as $s \to 1$, that the constant $\epsilon_0$ will not goes to $0$ as $s \to 1$. 

\end{remark}

Under a framework of global solutions prescribing data at infinity and global barriers, which are set up to avoid complexity of dealing with issues from the boundary data for non-local equations, the following theories for fractional Monge-Amp\'{e}re equations also work for fractional k-Hessian equations: 

(c) Existence of solutions. (Theorem \ref{existence})

(d) Semiconcavity and Lipschitz continuity of solutions. (Theorem \ref{semiconcavity})

(e) The non-local fully nonlinear theory developed in \cite{EC1} \cite{EC2} applies, in particular the nonlocal Evans-Krylov theorem. 

\begin{theorem}\label{existence}
There exists a unique solution of 
$$ \left\{
\begin{aligned}
&F_{k,s}[u](x)=u(x)-\phi(x) \ \  \text{in}\ \  \mathbb{R}^n\\
&(u-\phi)(x)\to 0 \ \  \text{as}\ \  |x|\to \infty.
\end{aligned}
\right.
$$
\end{theorem}

\begin{theorem}\label{semiconcavity}
Assume $\phi$ is semiconcave and Lipschitz continuous, and let v be the solution of 
$$ \left\{
\begin{aligned}
&F_{k,s}[v](x)=v(x)-\phi(x) \ \  \text{in}\ \  \mathbb{R}^n\\
&(v-\phi)(x)\to 0 \ \  \text{as}\ \  |x|\to \infty.
\end{aligned}
\right.
$$
Then, v is Lipschitz continuous and semiconcave with the same constants as $\phi$.
\end{theorem}

\begin{remark}
The difference between fractional Monge-Amp\'{e}re operators and k-Hessian operators is the set of matrices $M$ among which we take infimum of fractional linear operators. In Monge-Amp\'{e}re, we consider the infimum among all positive symmetric matrices with determinant $n^{-n}$, and in k-Hessian, we consider the infimum among all positive symmetric matrices in the set $\mathcal{M}_k$ (which will be discussed in Section 2, Proposition \ref{kHessian}). Hence, we can apply the exact same proofs of existence and $C^{1,1}$ regularity in the fractional Monge-Amp\'{e}re case, which are carefully explained in section 4,5 and 6 in \cite{C1}, to prove Theorem \ref{existence} and Theorem \ref{semiconcavity} for our fractional k-Hessian equations. \\
\end{remark}

Thus by what we have proved in (b), that such operators are strictly elliptic, and $C^{1,1}$ estimates in (d), we can apply nonlocal Evans-Krylov theorem \cite{EC1} \cite{EC2}  to prove solutions of fractional 2-Hessian equations are $C^{2s+\alpha}$, and further classical, under the framework of global solutions prescribing data at infinity and global barriers.\\

\begin{remark}
The proof for strictly ellipticity of the operator is required to improve the $C^{1,1}$ regularity to $C^{2s+\alpha}$ regularity. Therefore, we only care about the case $1/2<s<1$ in Theorem \ref{main theorem}, or there is no improvement in the regularity. We also care what would happen as $s \to 1$, and in the Remark 1.6, we can see that Theorem \ref{main theorem} is stable as $s \to 1$.

\end{remark}


   \section{Notations and Preliminaries}
   \label{prelim}
In this section, we will first state some notations. And then we will discuss one important representation of Monge-Amp\'{e}re operator(Proposition \ref{MongeAmpere}). Next we will derive a similar representation for k-Hessian operator(Proposition \ref{kHessian}), show how we construct the set $\mathcal{M}_k$ in Definition \ref{def1}, and give the definition of fractional k-Hessian operator. 

Given a function $u$, we shall denote the second-order increment of $u$ at $x$ in the direction of $y$ as
\[
\delta(u, x, y) = u(x+y)+u(x-y)-2u(x),
\]
and fractional laplacian is defined as
\begin{equation*}
\begin{aligned}
-(-\Delta)^su(x)&=C_{n,s}P.V.\int_{\mathbb{R}^n}{\frac{u(y)-u(x)}{|x-y|^{n+2s}}dy}\\
& =\frac{C_{n,s}}{2}\int_{\mathbb{R}^n}{\frac{u(x+y)+u(x-y)-2u(x)}{|x-y|^{n+2s}}dy}.
\end{aligned}
\end{equation*}
And the constant $C_{n,s}$ is a normalization constant.\\

For square matrices, $A>0$ means positive definite and $A\geq 0$ positive semidefinite. We denote $\lambda_i(A)$ the eigenvalues of A, in particular $\lambda_{min}(A)$ and $\lambda_{max}(A)$  are the smallest and largest eigenvalues, respectively.

We shall denote the $n$th-dimensional ball of radius $r$ and center $x$ by $B^n_r(x)=\{y\in \mathbb{R}^n, |y-x|<r\}$, and the corresponding $(n-1)$-dimensional sphere by $\partial B^n_r(x)=\{y\in \mathbb{R}^n, |y-x|=r\}$. $\mathcal{H}^{n}$ stands for the n-dimensional Haussdorff measure.\\

Let $A \subset \mathbb{R}^n$ be an open set. We say that a function $u: A \to \mathbb{R}$ is semi-concave if it is continuous in $A$ and there exists a constant $SC\geq 0$ such that $\delta(u, x, y)\leq SC|y|^2$ for all $x,y \in \mathbb{R}^n$ such that the segment $[x-y, x+y] \subset A$. And the constant $SC$ is called a semi-concavity constant for u in A.
Alternatively, a function $u$ is semi-concave in $A$ with constant $SC$ if $u(x)-\frac{SC}{2}|x|^2$ is concave in $A$. Geometrically, this means that the graph of $u$ can be touched from above at every point by a
paraboloid of the type $a+<b,x>+\frac{SC}{2}|x|^2$.\\

We denote the constant $C_n^k=\frac{n!}{k!(n-k)!}$ for $n,k \in \mathbb{N}$ and $n \geq k$.\\

We can write Monge-Amp\'{e}re operator as a concave envelope of linear operators, that
\begin{proposition}
\label{MongeAmpere}
If $u$ is convex, then the Monge-Amp\'{e}re operator $f(D^2u)=(\det{D^2u})^{1/n}$ can be expressed as 
\[
f(D^2u)=(\det{D^2u})^{1/n}=\inf_{M\in \mathcal{M}}{L_Mu},
\]
where $\mathcal{M}$ is the set of all positive symmetric matrices with determinant $n^{-n}$, and the linear operator $L_Mu$ is defined by
\[
L_Mu=trace(MD^2u)=\Delta(u\circ \sqrt{M}).
\]
\end{proposition}

\begin{proof}[Proof of Proposition \ref{MongeAmpere}]
Let $A=D^2u(x)$ which is positive, and we consider Monge-Amp\'{e}re operator $f(A)=(\det{A})^{1/n}$ as a concave envelope of linear operators, that
\[
f(A)=\inf_{B \in \Gamma_n}{\{Df(B)(A-B)+f(B)\}},
\]
and $Df(B)$ is a linear operator mapping $\mathbb{R}^{n \times n}$ to $\mathbb{R}$, that
\[
Df(B)A=\lim_{\epsilon \to 0}{\frac{f(B+\epsilon A)-f(B)}{\epsilon}}.
\]
Since $f$ is homogeneous of degree 1, that for any $t>0$,
\[
f(tB)=tf(B),
\]
and we can prove
\[
Df(B)B=\lim_{\epsilon\to0}{\frac{f(B+\epsilon B)-f(B)}{\epsilon}}=f(B).
\]
Letting $E_{ij} \in \mathbb{R}^{n \times n}$ be the matrix with the $i,j$th entry being 1 and all other entries being 0, we can calculate
\[
Df(B)E_{ij}=\frac{1}{n}(\det{B})^{\frac{1}{n}-1}b_{ij}^*,
\]
where $b_{ij}^*$ is the $i,j$th entry of the cofactor matrix of B. Thus, by linearity,
\[
Df(B)A=Df(B)(a_{ij}E_{ij})=a_{ij}(\frac{1}{n}(\det{B})^{\frac{1}{n}-1}b_{ij}^*)=trace(AM^T),
\]
where
\[
M=M(B)=Df(B)=\frac{1}{n}(\det{B})^{\frac{1}{n}-1}b_{ij}^*.
\]
And by the property of cofactor matrix $B^*$ that $B^{-1}=(\det{B})^{-1}B^*$, we know
\[
\det{M}=n^{-n}.
\]
Therefore, by the bijection between matrices and cofactor matrices, without loss of generality, we can conclude that
\[
(\det{D^2u})^{1/n}=\inf_{M\in \mathcal{M}}{L_Mu}=\inf_{M\in \mathcal{M}}{trace(MD^2u)},
\]
where $\mathcal{M}$ is the set of all positive symmetric matrices with determinant $n^{-n}$. \\
\end{proof}

Monge-Amp\'{e}re operator is the n-Hessian operator. Thus, we try to find a similar way of representing the concave k-Hessian operator.\\

\begin{proposition}
\label{kHessian}
If $D^2u \in \Gamma_k$, then the k-Hessian operator 
\[
f_k(D^2u)=(\sum_{1\leq i_1<i_2<...<i_k \leq n}{\lambda_{i_1}\lambda_{i_2}...\lambda_{i_k}})^{1/k}
\]
is a concave envelope of linear operators, that
\[
f_k(D^2u)=\inf_{M\in \mathcal{M}_k}\{trace(MD^2u)\}\\.
\]
And a matrix $M \in \mathcal{M}_k$ if there exists a matrix $B \in \Gamma_k$, such that the $i,j$th entry of the matrix $M$ satisfies the following conditions:
\begin{equation}\label{Mii}
M_{ii}=\frac{1}{kf_k(B)^{k-1}} \sum_{\mbox{\tiny$\begin{array}{c}
i\leq j_1<j_2<...<j_{k-1}\leq n,\\
j_1,...,j_{k-1} \neq i \end{array}$}}\det{B_{(j_1,...,j_{k-1})}},
\end{equation}
where $B_{(j_1,...,j_{k-1})}$ denotes the submatrix of $B$ formed by choosing the $j_1,j_2,...,j_{k-1}$th rows and columns.\\
When $k \geq 3$,
\begin{equation}\label{Mij3}
M_{ij}=-\frac{1}{kf_k(B)^{k-1}} \sum_{\mbox{\tiny$\begin{array}{c}
i\leq j_1<j_2<...<j_{k-2}\leq n,\\
j_1,...,j_{k-2} \neq i,j \end{array}$}}\det{B_{(j,j_1,...,j_{k-2})(i,j_1,...,j_{k-2})}},
\end{equation}
where $B_{(j,j_1,...,j_{k-2})(i,j_1,...,j_{k-2})}$ denotes the submatrix of $B$ formed by choosing the $j,j_1,j_2,...,j_{k-2}$th rows and  $i,j_1,j_2,...,j_{k-2}$th columns.\\
And when $k=2$,
\begin{equation}\label{Mij2}
M_{ij}=-\frac{1}{2f_2(B)}b_{ji},
\end{equation}
where $b_{ji}$ denotes the $j,i$th entry of matrix $B$. \\
Moreover, for each $M \in \mathcal{M}_k$, $M$ is a positive symmetric matrix. \\
\end{proposition}
\begin{proof}[Proof of Proposition \ref{kHessian}]


Since $f_k$ is a concave function of $\lambda=(\lambda_1,\lambda_2,...,\lambda_n)\in \Gamma_k$, with $\lambda_j, j=1,2,...,n$ eigenvalues of matrix A, we can write
\[
f_k(A)=\inf_{B\in \Gamma_k}{\{Df_k(B)(A-B)+f_k(B)\}}.
\]
Here $Df_k(B): \mathbb{R}^{n \times n} \to \mathbb{R}$ is an operator defined by
\[
Df_k(B)A=\lim_{\epsilon \to 0}{\frac{f_k(B+\epsilon A)-f_k(B)}{\epsilon}}.
\]
Take a basis $\{E_{ij}\}_{i,j=1}^n$ of $\mathbb{R}^{n\times n}$, that $E_{ij}$ is a matrix with $i,j$ th entry being 1, and all other entries being 0, we can calculate that

\[
Df_k(B)E_{ii}=\frac{1}{kf_k(B)^{k-1}} \sum_{\mbox{\tiny$\begin{array}{c}
i\leq j_1<j_2<...<j_{k-1}\leq n,\\
j_1,...,j_{k-1} \neq i \end{array}$}}\det{B_{(j_1,...,j_{k-1})}},
\]
where $B_{(j_1,...,j_{k-1})}$ denotes the submatrix of $B$ formed by choosing the $j_1,j_2,...,j_{k-1}$th rows and columns.\\
When $k \geq 3$,
\[
Df_k(B)E_{ij}=-\frac{1}{kf_k(B)^{k-1}} \sum_{\mbox{\tiny$\begin{array}{c}
i\leq j_1<j_2<...<j_{k-2}\leq n,\\
j_1,...,j_{k-2} \neq i,j \end{array}$}}\det{B_{(j,j_1,...,j_{k-2})(i,j_1,...,j_{k-2})}},
\]
where $B_{(j,j_1,...,j_{k-2})(i,j_1,...,j_{k-2})}$ denotes the submatrix of $B$ formed by choosing the $j,j_1,j_2,...,j_{k-2}$th rows and  $i,j_1,j_2,...,j_{k-2}$th columns.\\
And when $k=2$,
\[
Df_2(B)E_{ij}=-\frac{1}{2f_2(B)}b_{ji},
\]
where $b_{ji}$ denotes the $j,i$th entry of matrix $B$.\\

Define a matrix $M \in \mathbb{R}^{n \times n}$ where
\[
M_{ii}=Df_k(B)E_{ii},
\]
\[
M_{ij}=Df_k(B)E_{ij}.
\]
And we write $M=M(B)=Df_k(B)$ to denote this relation between matrix $B$ and $M$. Then for any matrix $A \in \mathbb{R}^{n \times n}$, $A=a_{ij}E_{ij}$, by linearity, 
\[
Df_k(B)A=a_{ij}Df_k(B)E_{ij}=a_{ij}M_{ij}=trace(AM^T).
\]
Moreover, since $f_k$ is homogeneous of degree 1, so
\[
Df_k(B)B=f_k(B).
\]
And therefore,
\begin{equation*}
\begin{aligned}
f_k(A)&=\inf_{B\in \mathbb{R}^{n\times n}}{\{Df_k(B)(A-B)+f_k(B)\}} \\
&=\inf_{B\in \mathbb{R}^{n\times n}}{\{trace(AM^T),M=M(B)\}} \\
&=\inf_{M \in \mathcal{M}_k}{\{trace(AM^T)\}}.
\end{aligned}
\end{equation*}
We can write the set
\begin{equation*}
\begin{aligned}
\mathcal{M}_k=\{& M\in \mathbb{R}^{n \times n}, \ \text{exist}  \ B \in \Gamma_k, M=Df_k(B)=M(B)\}.
\end{aligned}
\end{equation*}
Actually, $\mathcal{M}_k$ is the image set of all matrices in $\Gamma_k$ under the mapping 
\[
B \mapsto M=M(B)=Df_k(B),
\]
and a matrix $M \in \mathcal{M}_k$ if there exists a matrix $B \in \Gamma_k$ such that with entries of $M$ satisfying \eqref{Mii}, \eqref{Mij3} (when $k \geq 3$) or \eqref{Mij2} (when $k=2$).\\

Without loss of generality, we can assume $M$ to be symmetric. Assume the matrix $B$ has eigenvalues $\lambda_1, \lambda_2,...,\lambda_n$ and since $f_k$ is invariant under orthonormal transformation, that $f_k(B)=f_k(Q^T\Lambda Q)$, with $\Lambda$ be the diagonal matrix with diagonal entries $\lambda_1, \lambda_2,...,\lambda_n$. Then the matrix $M=Df_k(B)$ has same eigenvalues as $Df_k(\Lambda)$. And since $f_k$ is elliptic, thus the $i$th diagonal entry of $Df_k(\Lambda)$ satisfies
\[
(Df_k(\Lambda))_{ii}=\lim_{\epsilon \to 0}{\frac{f_k(\Lambda+\epsilon E_{ii})-f_k(\Lambda)}{\epsilon}}>0.
\]
Therefore, if $B \in {\Gamma_k}$, then $M=Df_k(B)$ is a positive matrix. In particular, if $B=diag\{\sigma_1,\sigma_2,...,\sigma_n\}$ and $f_k(B)=1$, then
\[
M=Df_k(B)=diag\{\lambda_1,\lambda_2,...,\lambda_n\}
\]
with
\[
\lambda_i=\frac{1}{k}(\sum_{1\leq i_1<...<i_{k-1}\leq n, i_j \neq i}{\sigma_{i_1}\sigma_{i_2}...\sigma_{i_{k-1}}}).
\]
\end{proof}

From Proposition \ref{kHessian}, we write
\begin{equation*}
\begin{aligned}
f_k(D^2u(x))&=\inf_{M \in \mathcal{M}_k}{\{trace(D^2u(x)M^T)\}}\\
&=\inf_{M \in \mathcal{M}_k}{\{trace(\sqrt{M}^TD^2u(x)\sqrt{M})\}} \\
&=\inf_{M \in \mathcal{M}_k}{\{\Delta(u\circ \sqrt{M})(x)\}},
\end{aligned}
\end{equation*}
Then it is natural to give Definition \ref{def1} of fractional k-Hessian operator by writing
\[
F_{k,s}[u](x)=\inf_{M \in \mathcal{M}_k}{\{-C_{n,s}^{-1}(-\Delta)^{s}(u\circ \sqrt{M})(x)\}}.
\]
   
   \section{The main mathematical results}
   \label{main}

In this section we will prove Theorem \ref{main theorem}, that when $k=2$, the infimum in the definition \eqref{eq1} of $F_{k,s}$, cannot be realized by matrices that are too degenerate, which proves that the fractional 2-Hessian operator is locally uniformly elliptic. Then we can apply theories for uniformly elliptic non-local operators such as Evans-Krylov theorem to our fractional 2-Hessian operators, to get $C^{2,\alpha}$ estimates for global solutions prescribing data at infinity and global barriers, and further to prove that such solutions are classical.\\

Our aim is to prove that as $\epsilon \to 0$,
\[
\inf_{M \in \mathcal{M}_k}{\{-C_{n,s}^{-1}(-\Delta)^{s}(u\circ \sqrt{M})(x),  \lambda_{min}(M)=\epsilon\}} \to \infty.
\]
And this will show that the infimum cannot be realized by matrices that are too degenerate, which is the result of Theorem \ref{main theorem}. To prove this, we want to consider the integral on $\partial B^n_r(0)$ as an average of integrals on $\partial B^{n-1}_r(0)$. Consider a unit vector
\[
\tilde{e}(\theta)=(0,0,...,0,\sin \theta, \cos \theta),
\]
with $\theta \in (-\pi/2,\pi/2]$. Then
\[
span\{\tilde{e}(\theta)\}^{\perp}=span \{ \tilde{e}_1,\tilde{e}_2,...,\tilde{e}_{n-1}\}
\]
with $\tilde{e}_j, j=1,2,...,n-1$ be the orthonormal basis of the $n-1$ dimensional perpendicular space. Especially, we can consider
\[
\tilde{e}_j=(0,0,...,0,1,0,...,0) \ \  j=1,2,...,n-2,
\]
and
\[
\tilde{e}_{n-1}=(0,0,...,0,\cos \theta,-\sin \theta).
\]
Then for any $y \in \partial B^n_r(0)$, and $y \perp \tilde{e}(\theta)$, we can write $y=(y_1,y_2,...,y_n)$ as 
\[
y=z_1\tilde{e}_1+z_2\tilde{e}_2+...+z_{n-1}\tilde{e}_{n-1},
\]
and therefore,
\[
y_j=z_j, \ \  j=1,2,...,n-2,
\]
\[
 y_{n-1}=z_{n-1}\cos \theta,
\]
\[
y_n=-z_{n-1}\sin \theta.
\]

Now let $M \in \mathcal{M}_2$, $\sqrt{M}^{-1}=diag\{\lambda_1,\lambda_2,...,\lambda_n\}$, assume 
\[
\lambda_1 \leq \lambda_2 \leq ...\leq \lambda_n=\epsilon^{-1/2},
\]
and write integral in $\mathbb{R}^n$ as an average of $(n-1)$-dimensional subspace perpendicular to $\tilde{e}(\theta)$, $-\pi/2<\theta\leq \pi/2$, that
\begin{equation*}
\begin{aligned}
I&=\prod{\lambda_j}\int_{\mathbb{R}^n}{\frac{u(y)-u(0)}{(\lambda_1^2y_1^2+...+\lambda_n^2y_n^2)^{\frac{n+2s}{2}}}dy}\\
&=\prod{\lambda_j}\int_{-\pi/2}^{\pi/2}{\int_{0}^{\infty}{\int_{x\in \partial B^{n-1}_1(0),x\perp \tilde{e}(\theta)}{\frac{u(r(x_1\tilde{e}_1+...+x_{n-1}\tilde{e}_{n-1}))-u(0)}{r^{1+2s}(\lambda_1^2x_1^2+...+(\lambda_{n-1}^2\cos^2\theta+\lambda_n^2\sin^2 \theta)x_{n-1}^2)^{\frac{n+2s}{2}}}dx}dr} d\theta}\\
&=\prod{\lambda_j}\int_{-\theta_0}^{\theta_0}{...d\theta}+\prod{\lambda_j}\int_{\theta_0<|\theta|\leq \pi/2}{...d\theta}\\
&=I_1+I_2.\\
\end{aligned}
\end{equation*}
Our aim is to show that as $\epsilon \to 0$, $I_1 \to \infty$(Proposition \ref{I1}), and $I_2 \geq 0$(Proposition \ref{I2}).\\

We need to prove the fractional laplacian of the restriction of u to any $(n-1)$-dimensional subspace is positive and bounded from above:

\begin{proposition}
\label{prop1}
Assume that $u$ satisfies all conditions in Theorem \ref{main theorem}, then
\[
0<\mu_0\leq (1-s)\int_{\mathbb{R}^{n-1}}{\frac{u(z_1e_1+z_2e_2+...+z_{n-1}e_{n-1})-u(0)}{|\bar{z}|^{n-1+2s}}d\bar{z}}\leq \mu_1.
\]
for each orthonormal basis $\{e_j\}_{j=1}^{n-1}$ of $\mathbb{R}^{n-1}$, where
\[
\mu_0=\mu_0(\eta_0,n,s,L,SC)
\]
given by \eqref{mu0}, 
and
\[
\mu_1=\mu_1(n,s,L,SC)
\]
given by \eqref{mu1}.
\end{proposition}

\begin{proposition}\label{I1}
Assume that $u$ satisfies all conditions in Theorem \ref{main theorem}. When $M \in \mathcal{M}_2$, $\sqrt{M}^{-1}=diag\{\lambda_1,\lambda_2,...,\lambda_n\}$ and $\lambda_{min}(M)=\epsilon$, the integral
\begin{equation*}
\begin{aligned}
I_1&=\prod{\lambda_j}\int_{-\theta_0}^{\theta_0}{\int_{0}^{\infty}{\int_{x\in \partial B^{n-1}_1(0),x\perp \tilde{e}(\theta)}{\frac{u(r(x_1\tilde{e}_1+...+x_{n-1}\tilde{e}_{n-1}))-u(0)}{r^{1+2s}(\lambda_1^2x_1^2+...+(\lambda_{n-1}^2\cos^2\theta+\lambda_n^2\sin^2 \theta)x_{n-1}^2)^{\frac{n+2s}{2}}}dx}dr}d\theta}\\
& \geq \frac{C_4\mu_0}{1-s} \epsilon^{-s}.
\end{aligned}
\end{equation*}
Here $C_4=C_4(n,s,\eta_0,L,SC)$ is given by \eqref{C4}.
\end{proposition}

\begin{proposition}\label{I2}
Assume that $u$ satisfies all conditions in Theorem \ref{main theorem}. For each $M \in \mathcal{M}_2$, $\sqrt{M}^{-1}=diag\{\lambda_1,\lambda_2,...,\lambda_n\}$, the integral
\[
\prod{\lambda_j}\int_{\bar{y}\in \mathbb{R}^{n-1}}{\frac{u(y_1,y_2,...,y_{n-1},0)-u(0)}{(\lambda_1^2y_1^2+...+\lambda_{n-1}^2y_{n-1}^2)^{\frac{n+2s-1}{2}}}d\bar{y}} \geq 0.
\]
And this shows
\begin{equation*}
\begin{aligned}
I_2&=\prod{\lambda_j}\int_{|\theta|\geq \theta_0}{\int_{0}^{\infty}{\int_{x\in \partial B^{n-1}_1(0),x\perp \tilde{e}(\theta)}{\frac{u(r(x_1\tilde{e}_1+...+x_{n-1}\tilde{e}_{n-1}))-u(0)}{r^{1+2s}(\lambda_1^2x_1^2+...+(\lambda_{n-1}^2\cos^2\theta+\lambda_n^2\sin^2 \theta)x_{n-1}^2)^{\frac{n+2s}{2}}}dx}dr}d\theta}\\
&\geq 0
\end{aligned}
\end{equation*}

\end{proposition} 

Proposition \ref{I1} and Proposition \ref{I2} together prove the main theorem:
\begin{proof}[Proof of Theorem \ref{main theorem}]
Let $P$ be an orthogonal matrix such that 
\[
P^T\sqrt{M}^{-1}P=J=diag\{\lambda_1,...,\lambda_n\}.
\]
and $M \in \mathcal{M}_2$, with $\lambda_{\min}(M)=\epsilon$. then by Proposition \ref{I1} and Proposition \ref{I2},
\begin{equation}\label{degenerate}
\begin{aligned}
&\int_{\mathbb{R}^n}{\frac{u(y)-u(0)}{|\sqrt{M}^{-1}y|^{n+2s}}\det{\sqrt{M}^{-1}}dy}\\
&=\prod_{j=1}^{n}{\lambda_j}\int_{\mathbb{R}^n}{\frac{u(y)-u(0)}{|\lambda_1^2y_1^2+...+\lambda_n^2y_n^2|^{(n+2s)/2}}dy}\\
& = I_1+I_2\\
& \geq \frac{C_4\mu_0}{1-s} \epsilon^{-s}+0= \frac{C_4\mu_0}{1-s} \epsilon^{-s}.
\end{aligned}
\end{equation}
\ \\
Also, since $I \in \Gamma_2$, so
\[
M_0=Df_2(I)=\sqrt{\frac{n-1}{2n}}I\in \mathcal{M}_2.
\]
we can obtain 
\begin{equation}\label{example}
\begin{aligned}
F_{2,s}[u](0)&=\inf_{M \in \mathcal{M}_2}{\{P.V.\int_{\mathbb{R}^n}{\frac{u(y)-u(0)}{|\sqrt{M}^{-1}y|^{n+2s}}\det{\sqrt{M}^{-1}}dy}\}}\\
&\leq \int_{\mathbb{R}^n}{\frac{u(y)-u(0)}{|\sqrt{M_0}^{-1}y|^{n+2s}}\det{\sqrt{M_0}^{-1}}dy}\\
&=(\frac{n-1}{2n})^{s/2}\int_{\mathbb{R}^n}{\frac{u(y)-u(0)}{|y|^{n+2s}}dy}\\
&\leq(\frac{n-1}{2n})^{s/2}\frac{\mu_1}{1-s},
\end{aligned}
\end{equation}
here the last inequality is proved by Proposition \ref{prop1}.\\
Therefore, when $\epsilon$ is small enough, for instance, when
\begin{equation*}
\epsilon <\sqrt{\frac{2n}{n-1}}C_4^{1/s}(\frac{\mu_0}{\mu_1})^{1/s},
\end{equation*}
we can see
\begin{equation}\label{inequality}
\frac{C_4\mu_0}{1-s} \epsilon^{-s} >(\frac{n-1}{2n})^{s/2}\frac{\mu_1}{1-s}.
\end{equation}
Now we take 
\begin{equation*}
\epsilon_0=\sqrt{\frac{n}{n-1}}C_4^{1/s}(\frac{\mu_0}{\mu_1})^{1/s} <\sqrt{\frac{2n}{n-1}}C_4^{1/s}(\frac{\mu_0}{\mu_1})^{1/s}.
\end{equation*}
Combining \eqref{degenerate},  \eqref{example} and \eqref{inequality}, we can obtain
\begin{equation*}
\begin{aligned}
&\inf_{M \in \mathcal{M}_2}{\{\frac{1}{2}\int_{\mathbb{R}^n}{\frac{\delta(u,0,y)}{|\sqrt{M}^{-1}y|^{n+2s}}\det{\sqrt{M}^{-1}}dy}, \lambda_{min}(M)\leq \epsilon_0 \}}\\
&\geq \frac{C_4\mu_0}{1-s} \epsilon_0^{-s} \\
&>(\frac{n-1}{2n})^{s/2}\frac{\mu_1}{1-s}\\
&\geq F_{2,s}[u](0).
\end{aligned}
\end{equation*}
Therefore,
\[
\inf_{M \in \mathcal{M}_2}{\{\frac{1}{2}\int_{\mathbb{R}^n}{\frac{\delta(u,0,y)}{|\sqrt{M}^{-1}y|^{n+2s}}\det{\sqrt{M}^{-1}}dy}, \lambda_{min}(M)\leq \epsilon_0 \}}> F_{2,s}[u](0),
\]
and thus,
\[
F_{2,s}[u](0)=F_{2,s}^{\epsilon_0}[u](0),
\]
with 
\begin{equation}\label{theta}
\epsilon_0=\epsilon_0(n,s,\eta_0,S,LC)=\sqrt{\frac{n}{n-1}}C_4^{1/s}(\frac{\mu_0}{\mu_1})^{1/s}.
\end{equation}
And
\[
C_4=C_4(n,s,\eta_0,L,SC) 
\]
given by \eqref{C4}.
And this completes the proof for Theorem \ref{main theorem}.
\end{proof}

We will use the following lemmas to prove Proposition \ref{prop1}.\\
Take a matrix $B \in \Gamma_2$, that 
\[
B=diag \{\frac{2}{n-1}\epsilon,\frac{2}{n-1}\epsilon,...,\frac{2}{n-1}\epsilon,h(\epsilon)\}.
\]
And find $h(\epsilon)$ such that
\[
\sigma_2(B)=2\epsilon h(\epsilon)+\frac{2(n-2)}{n-1}\epsilon^2=1,
\]
and this means
\[
h(\epsilon)=\frac{1-\frac{2(n-2)}{n-1}\epsilon^2}{2\epsilon},
\]
and when $\epsilon$ is small enough, $h(\epsilon)\approx \frac{1}{2\epsilon}$.
Then as defined,
\begin{equation*}
\begin{aligned}
M(B)=\frac{1}{2\sigma_2(B)^{1/2}} diag \{\frac{2(n-2)}{n-1}\epsilon+h(\epsilon), \frac{2(n-2)}{n-1}\epsilon+h(\epsilon),...,\frac{2(n-2)}{n-1}\epsilon+h(\epsilon),2\epsilon\}.
\end{aligned}
\end{equation*}
So write $\sqrt{M}^{-1}=diag\{g(\epsilon),g(\epsilon),...,g(\epsilon),\epsilon^{-1/2} \}$, where
\[
g(\epsilon)=(\frac{n-2}{n-1}\epsilon+\frac{h(\epsilon)}{2})^{-1/2}.
\]
And we can see that $g(\epsilon)\approx 2\sqrt{\epsilon}$ when $\epsilon$ is very small. Then, since $M\in \mathcal{M}_2$, thus by the equation \eqref{eq1}
\begin{equation*}
\begin{aligned}
0<\frac{\eta_0}{1-s} &\leq \det(\sqrt{M}^{-1}) \int_{\mathbb{R}^n}{\frac{u(\bar{y},y_n)-u(0)}{(g(\epsilon)^2|\bar{y}|^2+\frac{1}{\epsilon}y_n^2)^{\frac{n+2s}{2}}}dy} \\
&=g(\epsilon)^{n-1}\epsilon^{-1/2}\int_{\mathbb{R}^n}{\frac{u(\bar{y},y_n)-u(\bar{y},0)}{(g(\epsilon)^2|\bar{y}|^2+\frac{1}{\epsilon}y_n^2)^{\frac{n+2s}{2}}}dy}+g(\epsilon)^{n-1}\epsilon^{-1/2}\int_{\mathbb{R}^n}{\frac{u(\bar{y},0)-u( 0)}{(g(\epsilon)^2|\bar{y}|^2+\frac{1}{\epsilon}y_n^2)^{\frac{n+2s}{2}}}dy}\\
&=J_1+J_2.
\end{aligned}
\end{equation*}

Lemma \ref{lemma1} will give an estimate of $J_1$ by semi-concavity and Lipschitz continuity of $u$.

\begin{lemma}\label{lemma1}
Assume that $u$ satisfies all conditions in Theorem \ref{main theorem}. Take  $\sqrt{M}^{-1}=diag\{g(\epsilon),g(\epsilon),...,g(\epsilon),\epsilon^{-1/2} \}$, then
\[
J_1=g(\epsilon)^{n-1}\epsilon^{-1/2}\int_{\mathbb{R}^n}{\frac{u(\bar{y},y_n)-u(\bar{y},0)}{(g(\epsilon)^2|\bar{y}|^2+\frac{1}{\epsilon}y_n^2)^{\frac{n+2s}{2}}}dy} \leq \epsilon^s C_1C_2,
\]
where $C_1=C_1(s,L,SC)$ and $C_2=C_2(n,s)$ are given by \eqref{C_1} and \eqref{C_2} respectively.
\end{lemma}

\begin{proof}
By Lipschitz continuity and semi-concavity of $u$,
\[
J_1 \leq g(\epsilon)^{n-1}\epsilon^{-1/2}\int_{\mathbb{R}^n}{\frac{\max \{2L|y_n|,SC|y_n|^2\}}{(g(\epsilon)^2|\bar{y}|^2+\frac{1}{\epsilon}y_n^2)^{\frac{n+2s}{2}}}dy},
\]
then we can do change of variables, letting
\[
z_n=y_n, z_j=\frac{y_j}{|y_n|}\sqrt{\epsilon}g(\epsilon), j=1,2,...,n-1.
\]
Then
\[
dz=dy\frac{1}{|y_n|^{n-1}}(\sqrt{\epsilon}g(\epsilon))^{n-1},
\]
and
\begin{equation*}
\begin{aligned}
I_1&\leq g(\epsilon)^{n-1}\epsilon^{-1/2}(\sqrt{\epsilon}g(\epsilon))^{1-n}\int_{\mathbb{R}^n}{\frac{\max \{2L|z_n|,SC|z_n|^2\}}{(1+|\bar{z}|^2)^{\frac{n+2s}{2}}|z_n|^{n+2s-n+1}\epsilon^{-(n+2s)/2}}d\bar{z}dz_n}\\
&\leq \epsilon^s \int_{\mathbb{R}}{\frac{\max \{2L|z_n|,SC|z_n|^2\}}{|z_n|^{1+2s}}dz_n} \int_{\mathbb{R}^{n-1}}{\frac{1}{(1+|\bar{z}|^2)^{\frac{n+2s}{2}}}d\bar{z}}\\
& \leq \epsilon^s C_1C_2.
\end{aligned}
\end{equation*}
Here we define two constants $C_1,C_2$ by following:
\begin{equation}\label{C_1}
C_1=C_1(s,L,SC)=\int_{\mathbb{R}}{\frac{\max \{2L|z_n|,SC|z_n|^2\}}{|z_n|^{1+2s}}dz_n},
\end{equation}
\begin{equation}\label{C_2}
C_2=C_2(n,s)=\int_{\mathbb{R}^{n-1}}{\frac{1}{(1+|\bar{z}|^2)^{\frac{n+2s}{2}}}d\bar{z}}.
\end{equation}
\end{proof}
Then Lemma \ref{lemma2} gives an estimate of the integral $J_2$.

\begin{lemma}\label{lemma2}
Assume that $u$ satisfies all conditions in Theorem \ref{main theorem}. Take  $\sqrt{M}^{-1}=diag\{g(\epsilon),g(\epsilon),...,g(\epsilon),\epsilon^{-1/2} \}$, then
\[
J_2=g(\epsilon)^{-2s}C_3\int_{\mathbb{R}^{n-1}}{\frac{u(\bar{z},0)-u(0)}{|\bar{z}|^{n+2s-1}}d\bar{z}},
\]
where $C_3=C_3(n,s)$ are given by \eqref{C3}.
\end{lemma}

\begin{proof}

\[
J_2=g(\epsilon)^{n-1}\epsilon^{-1/2}\int_{\mathbb{R}^n}{\frac{u(\bar{y},0)-u( 0)}{(g(\epsilon)^2|\bar{y}|^2+\frac{1}{\epsilon}y_n^2)^{\frac{n+2s}{2}}}dy}.
\]
By change of variables
\[
 z_j=y_j, \ \ j=1,2,...,n-1,
\]
\[
z_n=(\sqrt{\epsilon}g(\epsilon))^{-1}\frac{y_n}{|\bar{y}|},
\]
we will get
\[
dz=dy(\sqrt{\epsilon}g(\epsilon)|\bar{y}|)^{-1},
\]
and
\begin{equation*}
\begin{aligned}
J_2&=g(\epsilon)^{n-1}\epsilon^{-1/2}\int_{\mathbb{R}^n}{\frac{u(\bar{y},0)-u( 0)}{(g(\epsilon)^2|\bar{y}|^2+\frac{1}{\epsilon}y_n^2)^{\frac{n+2s}{2}}}dy}\\
&=(\sqrt{\epsilon}g(\epsilon))^{-1}g(\epsilon)^{n-1}\epsilon^{-1/2}\int_{\mathbb{R}^n}{\frac{u(\bar{z},0)-u(0)}{g(\epsilon)^{n+2s}|\bar{z}|^{n+2s-1}(1+z_n^2)^{\frac{n+2s}{2}} }d\bar{z}dz_n}\\
&=g(\epsilon)^{-2s}\int_{\mathbb{R}^{n-1}}{\frac{u(\bar{z},0)-u(0)}{|\bar{z}|^{n+2s-1}}d\bar{z}}\int_{\mathbb{R}}{\frac{1}{(1+z_n^2)^{\frac{n+2s}{2}} }dz_n}\\
&=g(\epsilon)^{-2s}C_3\int_{\mathbb{R}^{n-1}}{\frac{u(\bar{z},0)-u(0)}{|\bar{z}|^{n+2s-1}}d\bar{z}}.
\end{aligned}
\end{equation*}
Here we define a constant $C_3$ by the following:
\begin{equation}\label{C3}
C_3=C_3(n,s)=\int_{\mathbb{R}}{\frac{1}{(1+z_n^2)^{\frac{n+2s}{2}} }dz_n}.
\end{equation}
\end{proof}

Then combining the estimates for $J_1$ and $J_2$, we can prove Proposition \ref{prop1}:
\begin{proof}
From the equation, we can see
\[
0<\frac{\eta_0}{1-s}\leq J_1+J_2\leq \epsilon^sC_1C_2+g(\epsilon)^{-2s}C_3\int_{\mathbb{R}^{n-1}}{\frac{u(\bar{z},0)-u(0)}{|\bar{z}|^{n+2s-1}}d\bar{z}},
\]
and therefore,
\[
\int_{\mathbb{R}^{n-1}}{\frac{u(\bar{z},0)-u(0)}{|\bar{z}|^{n+2s-1}}d\bar{z}} \geq \frac{\frac{\eta_0}{1-s}-\epsilon^sC_1C_2}{C_3g(\epsilon)^{-2s}}.
\]
So we only need to take $\epsilon=\epsilon_1$ small enough such that
\[
\eta_0=2(1-s)C_1C_2\epsilon_1^s, 
\]
that
\[
\epsilon_1=(\frac{\eta_0}{2(1-s)C_1C_2})^{1/s},
\]
then
\[
\int_{\mathbb{R}^{n-1}}{\frac{u(\bar{z},0)-u(0)}{|\bar{z}|^{n+2s-1}}d\bar{z}} \geq \frac{\eta_0}{2(1-s)C_3}g(\epsilon_1)^{2s}.
\]
And we have calculated that
\[
g(\epsilon)=(\frac{1}{4\epsilon}+\frac{n-2}{2(n-1)}\epsilon)^{-1/2},
\]
thus
\[
g(\epsilon_1)^{2s}=(\frac{1}{4\epsilon_1}+\frac{n-2}{2(n-1)}\epsilon_1)^{-s},
\]
and we can define
\begin{equation}\label{mu0}
\mu_0=\mu_0(n,s,\eta_0,L,SC)=\frac{\eta_0}{2(1-s)C_3}g(\epsilon_1)^{2s},
\end{equation}
we obtain the estimates that
\[
\int_{\mathbb{R}^{n-1}}{\frac{u(\bar{z},0)-u(0)}{|\bar{z}|^{n+2s-1}}d\bar{z}} \geq \mu_0>0.
\]
And by doing any orthonormal transformation, we will be able to show if $\{e_j\}_{j=1}^{n-1}$ are orthonomarl basis of $\mathbb{R}^{n-1}$,
\[
\int_{\mathbb{R}^{n-1}}{\frac{u(z_1e_1+z_2e_2+...+z_{n-1}e_{n-1})-u(0)}{|\bar{z}|^{n+2s-1}}d\bar{z}}\geq \mu_0>0.
\]
On the other hand, if $u$ is Lipschitz continuous and semi-concave, then
\[
\int_{\mathbb{R}^{n-1}}{\frac{u(\bar{z},0)-u(0)}{|\bar{z}|^{n+2s-1}}d\bar{z}} \leq \int_{\mathbb{R}^{n-1}}{\frac{\max\{2L|\bar{z}|,SC|\bar{z}|^2\}}{|\bar{z}|^{n+2s-1}}d\bar{z}}\leq \frac{\mu_1}{1-s},
\]
with
\begin{equation}\label{mu1}
\mu_1=\mu_1(n,s,L,SC)=(1-s)\int_{\mathbb{R}^{n-1}}{\frac{\max\{2L|\bar{z}|,SC|\bar{z}|^2\}}{|\bar{z}|^{n+2s-1}}d\bar{z}}.
\end{equation}
\end{proof}

With the estimates in Proposition \ref{prop1}, now we start to prove Proposition \ref{I1}.
The main idea is that, when the smallest eigenvalue of matrix $M$ is close to 0, there will be some contraints on the eigenvalues and their square root inverse $\lambda_j$, since the matrix is in the set $\mathcal{M}_2$. We will prove that $\frac{1}{\lambda_{1}^{n+2s}}-\frac{1}{\lambda_{n-1}^{n+2s}}$ is very small compared with $\frac{1}{\lambda_1^{n+2s}}$. This and the lower bound in Proposition \ref{prop1} will make it possible to prove that the integral on a (n-1)-dimensional subspace,  close to $\{x_n=0\}$, is very large. 

\begin{proof}[Proof of Proposition \ref{I1}]
Our aim is to show that when $\epsilon$ is very small, $I_1 \geq C_4\mu_0 \epsilon^{-s}$.
We take $\theta_0=C\frac{\lambda_{n-1}}{\lambda_n}$ which is very small $(\theta_0\leq 2C\epsilon)$ and the constant $C$ depends on $\frac{\mu_1}{\mu_0}$, determined by \eqref{C}. When $|\theta|\leq \theta_0$,
\[
\lambda_{n-1}^2\cos^2\theta+\lambda_n^2\sin^2\theta \leq (1+C^2)\lambda_{n-1}^2
\]
and thus,
\[
(1-4C^2\epsilon^2)\lambda_1^2\leq \lambda_1^2x_1^2+...(\lambda_{n-1}^2\cos^2\theta+\lambda_n^2\sin^2 \theta)x_{n-1}^2 \leq (1+C^2)\lambda_{n-1}^2.
\]
Let
\[
A=\int_{0}^{\infty}{\int_{\{x\in  \partial B_1^{n-1}(0),u(rx)-u(0)>0\}}{\frac{u(rx)-u(0)}{r^{1+2s}}dx} dr}\geq 0,
\]
and
\[
B=\int_{0}^{\infty}{\int_{\{x \in  \partial B_1^{n-1}(0),u(rx)-u(0)\leq 0\}}{\frac{u(rx)-u(0)}{r^{1+2s}}dx} dr}\leq 0.
\]
Then by Proposition \ref{prop1},
\[
A+B\geq \frac{\mu_0}{1-s}>0,
\]
and 
\[
A\leq \frac{\mu_1}{1-s}.
\]
And we can have the following estimates
\begin{equation*}
\begin{aligned}
(1-s)I_1&=\prod{\lambda_j}\int_{-\theta_0}^{\theta_0}{\int_{0}^{\infty}{\int_{x\in \partial B_1^{n-1}(0)}{\frac{u(rx)-u(0)}{r^{1+2s}}\frac{1}{( \lambda_1^2x_1^2+...(\lambda_{n-1}^2\cos^2\theta+\lambda_n^2\sin^2 \theta)x_{n-1}^2)^{\frac{n+2s}{2}}}dx}dr}d\theta}\\
&\geq 2\theta_0\prod{\lambda_j}(A\frac{(1+C^2)^{-(n+2s)/2}}{\lambda_{n-1}^{n+2s}}+B \frac{1}{\lambda_1^{n+2s}})\\
&\geq (2C\lambda_1...\lambda_{n-2}\lambda_{n-1}^2)(\frac{\mu_0}{\lambda_1^{n+2s}}+\mu_1(\frac{(1+C^2)^{-(n+2s)/2}}{\lambda_{n-1}^{n+2s}}-\frac{1}{\lambda_1^{n+2s}}))\\
&\geq 2C\lambda_1^n(\frac{\mu_0+\mu_1(C_5-1)}{\lambda_1^{n+2s}}+C_5\mu_1(\frac{1}{\lambda_{n-1}^{n+2s}}-\frac{1}{\lambda_{1}^{n+2s}})).
\end{aligned}
\end{equation*}
Here 
\[
C_5=(1+C^2)^{-(n+2s)/2}
\]
and take constant $C$ such that
\[
\mu_0+\mu_1(C_5-1)\geq \mu_0/2,
\]
i.e.,
take
\begin{equation}\label{C}
C=\sqrt{ (1-\frac{\mu_0}{2\mu_1})^{\frac{-2}{n+2s}}-1}
\end{equation}
and 
\begin{equation}\label{C5}
C_5=1-\frac{\mu_0}{2\mu_1}.
\end{equation}

Now let's see what constraint we will have on $\lambda_j$ when the smallest eigenvalue of matrix $M \in \mathcal{M}_2$ is $\epsilon$. We want to show that the non-negative $\frac{1}{\lambda_{1}^{n+2s}}-\frac{1}{\lambda_{n-1}^{n+2s}}$ is very small compared with $\frac{1}{\lambda_1^{n+2s}}$.\\

Let $B=diag\{\sigma_1,\sigma_2,...,\sigma_n\} \in \Gamma_2$. Assume $\sigma_1 \leq \sigma_2\leq...\leq \sigma_n$, and $\sum{\sigma_i\sigma_j}=1$. Then $M=diag \{\eta_1,\eta_2,...,\eta_n\}=Df_2(B)$, with $\eta_1 \geq \eta_2\geq...\geq \eta_n=\epsilon$, and
\[
\eta_j=\frac{1}{2}(\sum_i{\sigma_i}-\sigma_j).
\]
Then
\[
\sigma_1+\sigma_2+...+\sigma_{n-1}=2\epsilon=2\eta_n.
\]
Let $Q=\sigma_2+\sigma_3+...+\sigma_{n-1}$. Then $Q>\frac{2(n-2)}{n-1}\epsilon$. And since $\sum{\sigma_i\sigma_j}=1$, so
\begin{equation*}
\begin{aligned}
1&=\sigma_n(\sum_{i=1}^{n-1}{\sigma_i})+\sum_{1\leq i<j\leq n-1}{\sigma_i\sigma_j}\\
&=\sigma_n(2\epsilon)+\sigma_1(Q)+\sum_{2\leq i<j\leq n-1}{\sigma_i\sigma_j}\\
&\leq 2\epsilon \sigma_n+(2\epsilon-Q)Q+\frac{Q^2}{2}\\
&=2\epsilon \sigma_n+2\epsilon Q-\frac{Q^2}{2}.
\end{aligned}
\end{equation*}
Then
\[
\sigma_n \geq \frac{1+Q^2/2-2\epsilon Q}{2\epsilon}.
\]
And therefore
\[
\eta_1=\frac{1}{2}(Q+\sigma_n)\geq \frac{1+Q^2/2}{4\epsilon}.
\]
In addition, since $\sigma_1=2\epsilon-Q$, and $\sigma_{n-1}=2\epsilon-\sigma_1-\sigma_2-...-\sigma_{n-1}\leq 2\epsilon-(n-2)\sigma_1$, so
\[
0\geq \sigma_1-\sigma_{n-1}\geq (2n-4)\epsilon-(n-1)Q,
\]
and this means
\[
\eta_{n-1}-\eta_1 \geq (2n-4)\epsilon-(n-1)Q.
\]
Thereforem we can calculate
\[
\frac{1}{\lambda_{n-1}^{n+2s}}-\frac{1}{\lambda_{1}^{n+2s}}\geq \frac{n+2s}{2}\frac{1}{\lambda_1^{n+2s-2}}(\frac{1}{\lambda_{n-1}^{2}}-\frac{1}{\lambda_{1}^{2}})\geq \frac{n+2s}{2}\frac{1}{\lambda_1^{n+2s-2}}( (2n-4)\epsilon-(n-1)Q).
\]
Therefore,
\begin{equation*}
\begin{aligned}
(1-s)I_1&\geq 2C\lambda_1^n(\frac{\mu_0+\mu_1(C_5-1)}{\lambda_1^{n+2s}}+C_5\mu_1(\frac{1}{\lambda_{n-1}^{n+2s}}-\frac{1}{\lambda_{1}^{n+2s}}))\\
& \geq \frac{C\mu_0}{2}\eta_1^s+ \frac{C\mu_0}{2}\eta_1^s+C_5\mu_1(n+2s)\eta_1^{s-1}( (2n-4)\epsilon-(n-1)Q)\\
& \geq \frac{C\mu_0}{2}\epsilon^{-s}+\eta_1^{s-1}(\frac{C\mu_0}{2}\eta_1+C_6\epsilon-C_7Q)\\
& \geq \frac{C\mu_0}{2}\epsilon^{-s}+\eta_1^{s-1}(\frac{C\mu_0}{2}\frac{1+Q^2/2}{4\epsilon}+C_6\epsilon-C_7Q)\\
&\geq \frac{C\mu_0}{2}\epsilon^{-s}+\eta_1^{s-1}(\frac{C\mu_0}{8\epsilon}+(\sqrt{\frac{C\mu_0}{16\epsilon}}Q-C_7\sqrt{\frac{4\epsilon}{C\mu_0}})^2+C_6\epsilon-C_7^2\frac{4\epsilon}{C\mu_0})\\
&\geq \frac{C\mu_0}{2}\epsilon^{-s}+0\\
&\geq C_4\mu_0\epsilon^{-s},
\end{aligned}
\end{equation*}
when $\epsilon>0$ very small, and
\begin{equation}\label{C4}
C_4=C_4(n,s,L,SC,\eta_0)=\frac{C}{2}=\frac{1}{2}\sqrt{ (1-\frac{\mu_0}{2\mu_1})^{\frac{-2}{n+2s}}-1}.
\end{equation}
\end{proof}

Then we want to prove Proposition \ref{I2} by contradiction:
\begin{proof}[Proof of Proposition \ref{I2}]
Assume it is not true, then for some $M\in \mathcal{M}_2$, there exists a positive constant $A>0$ such that
\[
(1-s)\prod{\lambda_j}\int_{\bar{y}\in \mathbb{R}^{n-1}}{\frac{u(y_1,y_2,...,y_{n-1},0)-u(0)}{(\lambda_1^2y_1^2+...+\lambda_{n-1}^2y_{n-1}^2)^{\frac{n+2s-1}{2}}}d\bar{y}}=-A<0.
\]
Then since $M=diag\{\eta_1,...,\eta_n\} \in \mathcal{M}_2$, there exists $B=diag\{\sigma_1,...,\sigma_n\}\in \Gamma_2$. WLOG we require $\sum{\sigma_i\sigma_j}=1$. Then we can see $M=Df_2(B)$ and
\[
\eta_j=\frac{1}{2}(\sum_{i\neq j}{\sigma_i}).
\]
Take another matrix $\tilde{B} \in \Gamma_2$, that $\tilde{B}=diag \{\tilde{\sigma}_1,\tilde{\sigma}_2,...,\tilde{\sigma}_n\}$, and let
\[
\tilde{\sigma}_j=t_j\sigma_j, j=1,2,...,n-1
\]
and
\[
\tilde{\sigma}_n=g(t)\sigma_n.
\]
Given any $t>0$ every small, first find $n$ unknowns $t_1,t_2,...,t_{n-1},f(t)$ such that the following $n$ equations are satisfied:
\[
\frac{\tilde{\eta}_j}{\eta_j}=\frac{(\sum_{1\leq i\leq n-1}{t_i\sigma_i})-t_j\sigma_j+g(t)\sigma_n}{(\sum{\sigma_i})-\sigma_j}=\frac{1}{t}, j=1,2,...,n-1;
\]
and
\[
1=\sum\{\tilde{\sigma}_i\tilde{\sigma}_j\}=\sum_{1\leq i<j \leq n-1}{t_it_j\sigma_i\sigma_j}+g(t)\sigma_n(t_1\sigma_1+t_2\sigma_2+...+t_{n-1}\sigma_{n-1}).
\]
The last equation means
\[
g(t)=\frac{1-\sum_{1\leq i<j \leq n-1}{t_it_j\sigma_i\sigma_j}}{\sigma_n(t_1\sigma_1+t_2\sigma_2+...+t_{n-1}\sigma_{n-1})},
\]
and as $t,t_j \to 0$, $g(t) \to \infty$ if $\sigma_n>0$. And if $\sigma_n<0$, then $g(t)<0$ but still we will have $\tilde{\sigma}_n=g(t)\sigma_n$ positive and goes to $\infty$. Then
\[
\tilde{M}=Df_2(\tilde{B})=diag\{\tilde{\eta}_1,\tilde{\eta}_2,...,\tilde{\eta}_n\}
\]
with
\[
\tilde{\eta}_j=\frac{1}{t}\eta_j, j=1,2,...,n-1;
\]
\[
\tilde{\eta}_n=\frac{t_1\sigma_1+t_2\sigma_2+...+t_{n-1}\sigma_{n-1}}{\sigma_1+\sigma_2+...+\sigma_{n-1}}=h(t)\eta_n
\]
And as $t\to 0$,
\[
h(t) \to 0.
\]
Then
\[
\tilde{\lambda}_j=\sqrt{t}\lambda_j, j=1,2,...,n-1;
\]
and
\[
\tilde{\lambda}_n=h(t)^{-1/2}\lambda_n.
\]
Now since $\tilde{M} \in \mathcal{M}_2$ as well, therefore it satisfies the equation
\begin{equation*}
\begin{aligned}
0<\frac{\eta_0}{1-s}&\leq \prod{\tilde{\lambda}_j}\int_{\mathbb{R}^n}{\frac{u(y)-u(0)}{(\tilde{\lambda}_1^2y_1^2+...+\tilde{\lambda}_n^2y_n^2)^{(n+2s)/2}}dy}\\
& \leq \prod{\tilde{\lambda}_j}\int_{\mathbb{R}^n}{\frac{u(\bar{y},y_n)-u(\bar{y},0)}{(\tilde{\lambda}_1^2y_1^2+...+\tilde{\lambda}_n^2y_n^2)^{(n+2s)/2}}dy}+\prod{\tilde{\lambda}_j}\int_{\mathbb{R}^n}{\frac{u(\bar{y},0)-u(0)}{(\tilde{\lambda}_1^2y_1^2+...+\tilde{\lambda}_n^2y_n^2)^{(n+2s)/2}}dy}\\
&=P_1+P_2.\\
\end{aligned}
\end{equation*}
Define $\lambda=\min \{\lambda_1,...,\lambda_n\}>0$, first we can calculate $P_1$：
\begin{equation*}
\begin{aligned}
P_1& \leq t^{(n-1)/2}h(t)^{-1/2}\prod{\lambda_j}\int_{\mathbb{R}^n}{\frac{\max \{2L|y_n|,SC|y_n|^2\}}{(t(\lambda_1^2y_1^2+...+\lambda_{n-1}^2y_{n-1}^2)+\frac{1}{h(t)}\lambda_n^2y_n^2)^{(n+2s)/2}}dy}\\
&  \leq t^{(n-1)/2}h(t)^{-1/2}\lambda^{-n-2s}\prod{\lambda_j}\int_{\mathbb{R}^n}{\frac{\max \{2L|y_n|,SC|y_n|^2\}}{(t(y_1^2+...+y_{n-1}^2)+\frac{1}{h(t)}y_n^2)^{(n+2s)/2}}dy}.
\end{aligned}
\end{equation*}
Do change of variables
\[
z_j=\frac{y_j}{|y_n|}\sqrt{th(t)}, j=1,2,...,n-1
\]
and
\[
z_n=y_n,
\]
we can calculate
\[
P_1 \leq \lambda^{-n-2s}\prod{\lambda_j}h(t)^s\int_{\mathbb{R}}{\frac{\max \{2L|z_n|,SC|z_n|^2\}}{|z_n|^{1+2s}}dz_n} \int_{\mathbb{R}^{n-1}}{\frac{1}{(1+|\bar{z}|^2)^{\frac{n+2s}{2}}}d\bar{z}}.
\]
Calculating details are similar to the proof of Proposition \ref{prop1} and with definitions of \eqref{C_1} and \eqref{C_2} we know
\[
P_1 \leq  h(t)^s C(\lambda)C_1C_2.
\]
Then we calculate $P_2$, that 
\begin{equation*}
\begin{aligned}
P_2&=\prod{\tilde{\lambda}_j}\int_{\mathbb{R}^n}{\frac{u(\bar{y},0)-u(0)}{(\tilde{\lambda}_1^2y_1^2+...+\tilde{\lambda}_n^2y_n^2)^{(n+2s)/2}}dy}\\
&=t^{(n-1)/2}h(t)^{-1/2} \prod{\lambda_j} \int_{\mathbb{R}^n}{\frac{u(\bar{y},0)-u(0)}{(t(\lambda_1^2y_1^2+...+\lambda_{n-1}^2y_{n-1}^2)+\frac{1}{h(t)}\lambda_n^2y_n^2)^{(n+2s)/2}}dy}.
\end{aligned}
\end{equation*}
By change of variable, 
\[
z_j=y_j, j=1,2,...,n-1,
\]
\[
z_n=\frac{y_n}{(\lambda_1^2y_1^2+...+\lambda_{n-1}^2y_{n-1}^2)^{1/2}}\frac{\lambda_n}{\sqrt{th(t)}},
\]
we can calculate this integral
\begin{equation*}
\begin{aligned}
P_2&=t^{-s}\frac{1}{\lambda_n}\prod{\lambda_j}\int_{\mathbb{R}^{n-1}}{\frac{u(\bar{z},0)-u(0)}{(\lambda_1^2z_1^2+...+\lambda_{n-1}^2z_{n-1}^2)^{(n+2s-1)/2}}d\bar{z}}\int_{\mathbb{R}}{\frac{1}{(1+|z_n|^2)^{(n+2s)/2}}dz_n}\\
&=t^{-s}C_3\frac{1}{\lambda_n}\frac{-A}{1-s}.
\end{aligned}
\end{equation*}
Here
\[
C_3=C_3(n,s)=\int_{\mathbb{R}}{\frac{1}{(1+z_n^2)^{(n+2s)/2}}dz_n}
\]
is the same as in \eqref{C3}.
Then as $t \to 0$, since $A>0$ positive, and $h(t)\to 0$,
\[
P_1+P_2 \leq h(t)^sC(\lambda)C_1C_2-t^{-s}\frac{AC_3}{(1-s)\lambda_n} \to -\infty,
\]
and this contradicts
\[
P_1+P_2 =\prod{\tilde{\lambda}_j}\int_{\mathbb{R}^n}{\frac{u(y)-u(0)}{(\tilde{\lambda}_1^2y_1^2+...+\tilde{\lambda}_n^2y_n^2)^{(n+2s)/2}}dy} \geq \frac{\eta_0}{1-s}>0,
\]
which completes the proof of Proposition \ref{I2}.
\end{proof}

\section*{Acknowledgement }
\thanks{The author would like to thank her Ph.D. advisor, Professor Luis Caffarelli, for many valuable conversations on this project. She also want to thank Professor Sun-Yung Alice Chang, who shared her ideas on this topic and pointed out the k-cone problem. She is also grateful to many colleagues and friends, especially Hui Yu, who offered many helpful comments on this paper.}



\end{document}